\documentclass{article}
\usepackage{comment}
\usepackage{graphics}	
\usepackage[utf8]{inputenc}
\usepackage[english]{babel}										%
\usepackage{amsmath}
\usepackage{tikz}
\usetikzlibrary{cd}
\usepackage{enumerate}
\usepackage{bbm}
\usepackage{tikz-cd}
\tikzcdset{scale cd/.style={every label/.append style={scale=#1},
    cells={nodes={scale=#1}}}}
\usepackage{amssymb}
\usepackage{mathrsfs}
\tikzcdset{scale cd/.style={every label/.append style={scale=#1},
    cells={nodes={scale=#1}}}}
\usepackage{amsthm}
\newtheorem{theorem}{Theorem}
\newtheorem{lemma}[theorem]{Lemma}
\newtheorem{prop}[theorem]{Proposition}

\newtheorem{deff}[theorem]{Definition}
\newtheorem{cor}[theorem]{Corollary}
\newtheorem{rmk}[theorem]{Remark}
\newcommand{\perf}{\text{perf}}

\newcommand{\catC}{\mathscr{C}}

\newcommand{\calc}{\mathcal{C}}

\newcommand{\F}{\mathbb{F}}

\newcommand{\Hom}{\text{Hom}}

\newcommand{\Z}{\mathbb{Z}}
\newcommand{\C}{\mathbb{C}}
\newcommand{\Sym}{\text{Sym}}

\newcommand{\oo}{\mathcal{O}}

\newcommand{\End}{\mathrm{End}}

\newcommand{\id}{\mathrm{id}}

\newcommand{\Ob}{\mathsf{Ob}}

\newcommand{\triv}{\mathbbm{1}}

\newcommand{\catB}{\mathscr{B}}
\newcommand{\Com}{\underline{\text{Com}}}
\newcommand{\catD}{\mathscr{D}}
\author{Ville Nordström}
\title{Finite group actions on dg categories and Hochschild homology}
\date{}
\begin{document}
\maketitle
\begin{abstract}We prove a decomposition of the Hochschild homology groups of the equivariant dg category $\catC^G$ associated to a small dg category $\catC$ with direct sums on which a finite group $G$ acts. When the ground field is $\C$ this decomposition is related to a categorical action of $\text{Rep}(G)$ on $\catC^G$ and the resulting action of the representation ring $R_\C(G)$ on $HH_\bullet(\catC^G)$.
\end{abstract}

\section{Introduction}
Let $\catC$ be a dg category over a field of characteristic zero and $G$. If $G$ is a finite group which acts on $\catC$ via autoequivalences $\rho_g:\catC\to \catC$ we can consider the dg category $\catC^G$ of equivariant objects in $\catC$. When $\catC$ is the module category of a dg algebra $A$ and the group action comes from an action of $G$ on $A$ then it follows from \cite{Quddus} (\cite{Getzler-Jones} in the case of an ordinary algebra) that 
\begin{align}\label{decomposition}HH_\bullet(\catC^G)\cong \bigoplus_{[g]}HH_\bullet(\catC;\rho_g)^{C(g)}\end{align}
where $[g]$ denotes the conjugacy class of $g$ and $C(g)$ is the centralizer of $g$ in $G$. In this note we prove that this decomposition holds assuming only that $\catC$ is a small dg category that admits finite direct sums. To our knowledge this has not appeared in the literature. Even in situations where this decomposition was already known we believe proving it directly in the dg category setting is illuminating. The projection and inclusion maps
$$\begin{tikzcd}
HH_\bullet(\catC^G)\arrow[two heads, shift left]{r}{}	&HH_\bullet(\catC;\rho_g)^{C(g)}\arrow[hookrightarrow, shift left]{l}{}
\end{tikzcd}
$$
 associated to this decomposition arise very naturally in the dg category setting; they are induced by two dg functors
$$\catC^G\overset{For}{\to }\catC, \quad \catC\overset{S}{\to}\catC^G$$
and natural transformations 
$$S\circ \rho_g\cong S, \ \rho_g\circ For\cong For.$$

Moreover, in the dg category setting this decomposition has another interesting explanation when the ground field is $\C$. Namely, we explain how the representation ring of $G$ over $\C$, $R_\C(G)$ acts on $HH_\bullet(\catC^G)$. We then prove that the decomposition above matches up with the decomposition of the $R_\C(G)$-module 
$$HH_\bullet(\catC^G)=\bigoplus_{[g]} \varphi_{[g]}\cdot HH_\bullet(\catC^G)$$ 
where $\varphi_{[g]}\in R_\C(G)$
denotes the indicator function on the conjugacy class of $g$.

Finally we note that as a consequence of the decomposition (\ref{decomposition}) above we have that for any small dg category which admits finite sums, $\text{Sym}^n(HH_\bullet(\catC))$ appears naturally as a summand in $HH_\bullet(\text{Sym}^n(\catC))$. This is related to Conjecture 3.24 in \cite{Belmans-Fu-Krug}.

In \cite{Perry}, Perry establishes a decomposition analogous to (\ref{decomposition}) for Hochschild cohomology also using the functors $For$ and $S$ mentioned above. Apart from that however our approaches differ in the following way. They use adjoint properties of these functors and the description of Hochschild cohomology as a derived hom between bimodules. In general one cannot describe Hochschild homology as derived hom between bimodules so another approach is necessary. We use instead that Hochschild homology is functorial with respect to dg functors, something which is not true for cohomology, to produce the projections and inclusions associated to the decomposition (\ref{decomposition}).

\subsection{Summary of paper}
In section 2 we introduce group actions on dg categories. In section 3.1 we recall some basic facts about dg categories. In section 3.2 we recall the definition of Hochschild homology of dg categories with values in an endofunctor and prove a few simple lemmas. In section 3.3 we recall the Künneth formula for Hochschild homology. Although the Künneth formula is well known, we could not find a proof which both provides an explicit formula and works at the level of generality that we need so we include a proof as well. Finally in section 4.1 we state and prove our main result and in section 4.2 we describe an application of it.

\subsection{Conventions}
All vector spaces will be over a fixed ground field $k$ of characteristic zero. Sometimes we will need $k=\C$ and will then make this clear.

All functors between dg categories will be dg functors. By a natural transformation between two dg functors $F,F':\catC\to\catB$ we will always mean a closed degree zero natural transformations. An equivalence $F:\catC\to \catB$ between dg categories will mean a dg functor that admits an inverse $F^{-1}:\catB\to \catC$ such that $F\circ F^{-1}\cong \id_{\catB}$ and $F^{-1}\circ F\cong \id_\catC$ and these natural isomorphisms are required to be degree zero cycles.

Our convention on matrices is that they act on the left. So when we describe a morphism in some category between objects with specified direct sum decompositions by a matrix, the columns will be indexed by the summands of the source and the rows by the summands of the target.

\section{Finite group actions on dg categories} \label{Finite group actions on dg categories}
When talking about an action of a finite group $G$ on a dg category $\catC$ there are different levels of strictness that one can require. Here is the definition that we will work with.
\begin{deff}\label{groupaction}
An action of $G$ on $\catC$ is the data of an autoequivalence $\rho_g:\catC\overset{\sim}{\to} \catC$ for every $g\in G$, natural isomorphisms $\theta_{g,g'}:\rho_{g}\circ \rho_{g'}\overset{\sim}{\implies}\rho_{g'g}$ and $\eta:\rho_e\overset{\sim}{\implies} \id_\catC$ and these should satisfy the following two conditions.\\

i) For each $g\in G$ and $c\in \catC$ we have
$$(\theta_{e,g})_c=\eta_{\rho_g(c)}:\rho_e\rho_g(c)\to \rho_g(c)$$
and  
$$(\theta_{g,e})_c=\rho_g(\eta_{c}):\rho_g\rho_e(c)\to \rho_g(c).$$

ii) For all $g,h,k\in G$ we have commutative diagrams
$$\begin{tikzcd}\rho_g\rho_h\rho_k(c)\arrow{r}{\rho_g((\theta_{h,k})_c)}\arrow{d}{(\theta_{g,h})_{\rho_k(c)}}&\rho_{g}\rho_{kh}(c)\arrow{d}{(\theta_{g,kh})_c}\\
\rho_{hg}\rho_k(c)\arrow{r}{(\theta_{hg,k})_c} &\rho_{khg}(c)	
\end{tikzcd}
$$
\end{deff}
\begin{rmk}
	One could also consider a more lax notion of group action where each group element acts via a quasi equivalence and the equalities of morphisms in the two conditions above are replaced by specified homotopies. However, we don't know of interesting examples where such group actions arise. The type of group action we defined above on the other hand does arise. For example if $G$ acts on a smooth variety $X$ then we get a group action in the above sense on the dg category of perfect complexes $\perf(X)$ or on the dg category of matrix factorizations $MF(X,W)$ of some function $W\in \oo_X(X)^G$ by pulling back vector bundles.
\end{rmk}

Whenever we have a group action as above we can consider the category of $G$-equivariant objects in $\catC$. 
\begin{deff}\label{equivariantobject}
	Let $(\rho,\theta,\eta)$ be an action of a finite group $G$ on a small dg category $\catC$. The dg category of equivariant objects, denoted $\catC^G$ has objects 
	$$(c,(\alpha_g:c\overset{\sim}{\to}\rho_g(c))_{g\in G})$$
	where the $\alpha_g$ are isomorphisms that are required to satisfy the following condition. For any $c\in \catC$ and $g,h\in G$ the following diagram
	$$\begin{tikzcd}
		c\arrow{r}{\alpha_g}\arrow{d}{\alpha_{hg}}&\rho_g(c)\arrow{d}{\rho_g(\alpha_h)}\\
		\rho_{hg(c)}\arrow{r}{(\theta_{g,h}^{-1})_c}&\rho_g\rho_h(c)
	\end{tikzcd}$$
	commutes.
	
	A degree $n$ morphism between equivariant objects $\varphi:(c,\alpha)\to (c',\alpha')$ is a degree $n$ morphism in $\catC$ such that 
	$$\begin{tikzcd}
		c\arrow{r}{\varphi}\arrow{d}{\alpha_g}&c'\arrow{d}{\alpha_g'}\\
		\rho_g(c)\arrow{r}{\rho_g(\varphi)}&\rho_g(c').
	\end{tikzcd}$$
\end{deff}
If the category $\catC$ has finite direct sums we can define a tensor product $\otimes: \text{Rep}(G)\times \catC^G\to \catC^G$. Indeed, if $(c,\alpha)\in \catC^G$ and $(V,\rho_V)$ is a left $G$-module we define their tensor product $(V,\rho_V)\otimes (c,\alpha)$ as follows. Let $\{e_1,...,e_n\}$ be a basis for $V$. Then 
$$\bigoplus_{i=1}^nc\cdot e_i\in \catC$$
is the object underlying $(V,\rho_V)\otimes (c,\alpha)$. The equivariant structure is given in matrix form by $((\rho_V\otimes \alpha)_g)_{ij}:=\rho_V(g)_{ij}\cdot \alpha_g$. To verify that this is indeed an equivariant object we have to check that $(\alpha\otimes \rho_V)_g$ is indeed an isomorphism and that the condition from definition $\ref{equivariantobject}$ is satisfied. To see that it is an isomorphism we simply note that its inverse is given by the matrix $((\rho_g^{-1})_{ij}\alpha_g^{-1})_{ij}$. To see that it is an equivariant structure we note that
$$\sum_j \rho_V(h)_{ij}\rho_V(g)_{jk}\rho_g(\alpha_h)\alpha_g=\rho_V(hg)_{ik}\theta_{g,h}^{-1}\alpha_{hg}.$$

There is a forgetful functor $For:\catC^G\to \catC$. If $\catC$ has finite direct sums then there is also a functor in the other direction which we call the symmetrization functor and denote by $S$. It is defined by
$$S(c):=\Big(\bigoplus_{h\in G}\rho_h(c),\alpha_g=(\delta_{hg=h'}\theta_{g,h}^{-1})_{hh'}\Big).$$
\begin{prop}
	The symmetrization functor is left adjoint to the forgetful functor.
\end{prop}
\begin{proof}
Suppose $(c,\alpha)$ is an equivariant object and $c'$ is just an object. Define
$$\begin{tikzcd}\catC(c',c)\arrow{r}{}& \catC^G(S(c'),(c,\alpha))\\
\varphi\arrow[mapsto]{r}{}&(\alpha_g^{-1}\circ \rho_g(\varphi))_g.
\end{tikzcd}
$$

Now define
$$\begin{tikzcd}\catC^G(S(c'),(c,\alpha))\arrow{r}{}& \catC(c',c)\\
\psi\arrow[mapsto]{r}{}&\eta_c^{-1}\circ\alpha_e\circ\psi_e\circ \eta_{c'}.
\end{tikzcd}
$$
These are inverse bijections which proves the proposition.
\end{proof}

\begin{prop}\label{SFor}
	We have $S\circ For\cong k[G]\otimes (-)$ and $For\circ S\cong \bigoplus_g \rho_g$.
\end{prop}
\begin{proof}
	Let $(c,\alpha)$ be an equivariant object. Consider the following diagram, where the morphisms are written as matrices,
	$$\begin{tikzcd}
		k[G]\otimes c\arrow{rr}{(\delta_{g=g'^{-1}}\alpha_{g})_{gg'}}\arrow{d}{(\delta_{k'=hg'}\alpha_h)_{k'g'}}&&S\circ For(c)\arrow{d}{(\delta_{kh=g}\theta_{h,k}^{-1})_{kg}}\\
		\rho_h(k[G]\otimes c)\arrow{rr}{(\delta_{k=k'^{-1}}\rho_h(\alpha_k))_{kk'}}&&\rho_h(S\circ For(c)).
	\end{tikzcd}$$
The top horizontal arrow is an isomorphism in $\catC$ so if the diagram commutes it is also an isomorphism in $\catC^G$. We check that it commutes
\begin{align*}\sum_g\delta_{kh=g}\theta_{h,k}^{-1}\delta_{g=g'^{-1}}\alpha_g=\delta_{(g')^{-1}=kh}\theta_{h,k}^{-1}\alpha_{kh}=	\delta_{(g')^{-1}=kh}\rho_h(\alpha_k)\alpha_h=\\\sum_{k'}\delta_{k=k'^{-1}}\rho_h(\alpha_k)\delta_{k'=hg'}\alpha_h.
\end{align*}
\end{proof}

\section{Hochschild homology of dg categories}
\subsection{Some background on dg categories}\label{backgroundondg}
Here we give a brief summary of dg categories and their modules. We follow section 2.1 of \cite{Pol-Pol}.

A right module over dg category $\catC$ is a dg functor $M:\catC^{op}\to \underline{\text{Com}}(k)$ where the later is the dg category whose objects are chain complexes of vector spaces and whose morphism complexes are defined in the usual way. Similarly we define a left module as functors $\catC\to \Com(k)$. The collection of right (resp. left) dg modules themselves form a dg category in a natural way which we denote $mod-\catC$ (resp. $\catC-mod$). For any $c\in \Ob(\catC)$ we denote by $h_c$ the right module represented by $c$. The category $Z^0(mod-\catC)$ has the structure of an exact category where a sequence of modules
$$M\to M''\to M$$
is defined to be exact if for any $c\in \Ob(\catC)$ \textit{both} 
$$M(c)\to M''(c)\to M'(c)$$
is an exact sequence of complexes and 
$$H(M(c))\to H(M''(c))\to H(M'(c))$$
is an exact sequence of graded vector spaces.
\begin{lemma}\label{projectivemodules}
If $V$ is any chain complex of vector spaces, and $c\in\Ob(\catC)$ then the functor $V\otimes h_c$ is projective in the exact structure defined above.	 
\end{lemma}
\begin{proof}
Consider a diagram
$$\begin{tikzcd}
	&X\arrow[two heads]{d}{\pi}\\
	V\otimes h_c\arrow{r}{\varphi}&M.
\end{tikzcd}$$	
The map $\varphi$ corresponds to a degree $0$ cycle 
$$\hat{\varphi}\in \Hom_{Z^0(mod-\catC)}(h_c,\Hom_k^\bullet(V,M(-))=\Hom_{Ch(k)}(V,M(c))$$
where $\Hom_{Ch(k)}(A,B)$ denotes chain maps between complexes of vector spaces. Because $X(c)\to M(c)$ and $H(X(c))\to H(M(c))$ are both surjective the induced map
$\Hom_{Ch(k)}(V,X(c))\to \Hom_{Ch(k)}(V,M(c))$ is also surjective so we find a lift of $\hat{\varphi}$, call it $\hat{\psi}$. It corresponds to a map $\psi:V\otimes h_c\to X$ which makes the triangle above commute. 
\end{proof}

If $\catC$ and $\catB$ are two dg categories we can form their tensor product $\catC\otimes\catB$ which by definition is the dg category whose objects are pairs $(c,b)$ with $c\in \Ob(\catC)$ and $b\in \Ob(\catB)$ and the momorphism complexes are defined by
$$(\catC\otimes \catB)((c,b),(c',b')):= \catC(c,c')\otimes\catB(b,b').$$
When defining composition of morphisms in $\catC\otimes \catB$ we take the Koszul sign rule into account. 

A right module over $\catC\otimes \catB^{op}$ is also called a $\catB-\catC$-bimodule. If $M$ is a right $\catC$ module and $N$ is a left $\catC$ we define their tensor product $M\otimes_{\catC}N$ to be the cokernel of the following map
$$\bigoplus_{c,c'\in \catC}M(c)\otimes \catC(c',c)\otimes N(c)\to \bigoplus_{c\in \catC}M(c)\otimes N(c), \ m\otimes f\otimes n\mapsto mf\otimes n-m\otimes fn.$$
Note that $M\otimes_{\catC}N$ is just a chain complex. Similarly, if $M$ is a $\catB-\catC$-bimodule and $N$ is a $\catC-\catD$ bimodule we can define their tensor product $M\otimes_{\catC}N$ which will be a $\catB-\catD$-bimodule.

For any $\catC$ we have two diagonal bimodules. We denote by $\Delta^r_\catC$ the right $\catC\otimes \catC^{op}$ module defined by 
$$(c_1,c_2)\mapsto \catC(c_1,c_2).$$
We denote by $\Delta_{\catC}^l$ the left $\catC\otimes \catC^{op}$ module defined by 
$$(c_1,c_2)\mapsto \catC(c_2,c_1).$$

There is a standard way to resolve $\Delta_\catC^l$ by projective left $\catC\otimes \catC^{op}$-modules, called the bar resolution which we will denote by $\text{Bar}^\bullet(\catC)$. It is defined by 
$$\text{Bar}^n(\catC):=\Big((x,y)\mapsto\bigoplus_{c_0,...,c_{n}\in \catC}\catC(c_0,x)\otimes\catC(c_1,c_0)\otimes...\otimes \catC(y,c_{n})\Big)$$
for $n\geq 1$.
For each $n\geq 1$ we have the map of left $\catC\otimes \catC^{op}$ modules $d^n:\text{Bar}^n(\catC)\to \text{Bar}^{n-1}(\catC)$ defined by
$$f_0|f_1|\cdots|f_{n+1}\mapsto \sum_{i=0}^{n+1}(-1)^{i}f_0|\cdots |f_if_{i+1}|\cdots |f_{n+1}$$
and we have the map $d^0:\text{Bar}^0(\catC)\to \Delta_{\catC}^l$ defined by $f_0|f_1\mapsto f_0f_1$. Note that each $\text{Bar}^n(\catC)$ is projective as a left $\catC\otimes \catC^{op}$ modules by lemma \ref{projectivemodules}. The fact that this is a resolution in the exact structure follows from the fact that for any pair $(x,y)$ we have a contracting homotopy (which is not natural in $x$ and $y$ however)
$$\text{Bar}^{n}(\catC)\to\text{Bar}^{n+1}(\catC), \ f_0|\cdots |f_{n+1}\mapsto \id_x|f_0|\cdots |f_{n+1}.$$

\subsection{Hochschild homology}
We first define the standard complex which we will use to compute Hochschild homology
\begin{deff}
	Let $\catC$ be a small dg category and let $F:\catC\to \catC$ be an endofunctor. Set 
	$$C_n(\catC;F):=\bigoplus_{c_0,...,c_n\in \catC}\catC(c_1,F(c_0))\otimes \catC(c_2,c_1)\otimes...\otimes \catC(c_0,c_n).$$
	There is a differential $d_{2}:C_{n}(\catC;F)\to C_{n-1}(\catC;F)$ defined by 
	\begin{align*}d_2\big(a_0[a_1|\cdots |a_n]\big)&=a_0a_1[a_2|\cdots|a_n]\\
	&+\sum_{i=1}^{n-1}(-1)^{i}a_0[\cdots|a_ia_{i+1}|\cdots]\\
	&+(-1)^{n+|a_n|(|a_0|+...+|a_{n-1}|)}F(a_n)a_0[a_1|\cdots|a_{n-1}].
\end{align*}
Each $C_n(\catC;F)$ is itself a cohomologically graded chain complex $(C_n(\catC;F)^\bullet,d_1)$. The differentials $d_1$ and $d_2$ commute and therefore we get a total complex (cohomologically graded) by altering the sign of $d_1$ by $(-1)^n$ on $C_n(\catC;F)$ whose $k'$th term for $k\in \Z$ is
$$\bigoplus_{n-m=k}C_m(\catC;F)^n.$$
The homology of this complex is the Hochschild homology of $\catC$ with values in $F$ and it is denoted $HH_\bullet(\catC;F)$.
\end{deff}
\begin{rmk}
Different sign conventions appear for the Hochschild differential but they all give rise to isomorphic chain complexes and so we use this one as it is the least complicated. This sign convention appears for example in \cite{Lod}.
\end{rmk}

If we have two dg categories $\catC$, $\catC'$ and two endofunctors $F:\catC\to\catC$, $F':\catC'\to \catC'$ and $\phi:\catC\to \catC'$ is a functor and $\epsilon:\phi\circ F\implies F'\circ \phi$ is a natural transformation, then we get an induced map on Hochschild homology 
$$(\phi,\epsilon)_*:HH_\bullet(\catC;F)\to HH_\bullet(\catC';F')$$
defined at the chain level as
$$\begin{tikzcd}
	\catC(c_1,F(c_0))\otimes...\otimes \catC(c_0,c_n)\arrow{r}{}&\catC'(\phi(c_1),F'\circ \phi(c_0))\otimes...\otimes \catC'(\phi(c_0),\phi(c_n))\\
	f_0[f_1|\cdots |f_n]\arrow[mapsto]{r}{}&\epsilon_{c_0}\phi(f_0)\big[\phi(f_1)|\cdots |\phi(f_n)\big].
\end{tikzcd}$$
We will refer to the data of such a $\phi$ and $\epsilon$ as $(\phi,\epsilon):(\catC,F)\to (\catC',F')$.

\begin{lemma}\label{tracedecomposition}
Let $\catC, \ \catC'$ be dg categories, $F\in \End(\catC),F'\in \End(\catC')$be two endofunctors, $A_1,A_2:\calc\to \calc'$ two dg functors and $\eta_{ij}:A_j\circ F\implies F'\circ A_i$ natural transformations. Set 
$$\eta:=\begin{pmatrix}\eta_{11}&\eta_{12}\\\eta_{21}&\eta_{22}\end{pmatrix}:\begin{matrix}A_1\circ F\\\oplus\\A_2\circ F\end{matrix}\to \begin{matrix}F\circ A_1\\\oplus\\F\circ A_2\end{matrix}.$$
Then $$(A_1\oplus A_2,\eta)_*\simeq (A_1,\eta_{11})_*+(A_2,\eta_{22})_*$$
as maps $C_\bullet(\catC;F)\to C_\bullet (\catC';F')$.
\end{lemma}
\begin{proof}[Proof of lemma]
Let $\pi_i$ and $\iota_i$ be the natural projections and inclusions associated with the direct sum $A:=A_1\oplus A_2$. Note that for any morphism $f$ in $\catC$ we have $A(f)=\iota_1A_1(f)\pi_1+\iota_2A_2(f)\pi_2$. We have
\begin{align*}(A,\eta)_*(f_0[f_1|\cdots f_k])=\\
\eta \circ A(f_0)[A(f_1)|\cdots |A(f_k)]=\\
(\iota_1\eta_{11}+\iota_2\eta_{21})A_1(f_0)\pi_1[\iota_1A_1(f_1)\pi_1|\cdots |\iota_1A_1(f_k)\pi_1]+\\
(\iota_1\eta_{12}+\iota_2\eta_{22})A_2(f_0)\pi_2[\iota_2A_2(f_1)\pi_2|\cdots |\iota_2A_2(f_k)\pi_2]+\\
\text{mixed terms}.
\end{align*}
\begin{align}\label{terms}
\end{align}
For a term in (\ref{terms}) we say that the parity changes at position $l$ (for some $0\leq l\leq k$) if the $l+1$'th factor ends with a $\iota_2$ and the $l'$th factor begins with a $\pi_1$ or vice versa. When $l=k$ here then the $l+1$'th factor refers to the $0'th$ factor. So for example
$$\iota_1\eta_{11}A_1(f_0)\pi_1[\iota_2A_2(f_1)\pi_2|\iota_2A_2\pi_2]$$
changes parity at positions $0$ and $2$ and 
$$\iota_1\eta_{12}A_2(f_0)\pi_2[\iota_2A_2(f_1)\pi_2|\iota_2A_2\pi_2]$$
changes parity only at position 2.

For $n\geq 0$ and $i=1,2$ let $\phi^{n}_i$ be the sum of all the terms in $(\ref{terms})$ which changes parity at $n$ distinct positions and which has $A_i(f_0)$ in position zero. Then, each $\phi^n_i$ is a chain map and $(A_1\oplus A_2,\eta)_*=\sum_{n,i}\phi^n_i.$ We have
\begin{align*}
	\phi_i^0(f_0[f_1|\cdots f_k])=\iota_i\eta_{ii}A_i(f_0)\pi_i[\iota_i A_i(f_1)\pi_i|\cdots |\iota_i A_i(f_k)\pi_i]
\end{align*}
and this is homotopic to $(A_i,\eta_{ii})_*$ via the homotopy
\begin{align*}H_i(f_0[f_1|\cdots |f_k]):=\\\sum_{j=0}^k (-1)^{j} \eta_{ii} A_i(f_0)\pi_i[\iota_iA_i(f_1)\pi_i|\cdots |\iota_iA_i(f_j)\pi_i|\iota_i| A_i(f_{j+1})|\cdots |A_i(f_k)].\end{align*}
Next, for $n\geq 1$ we have $\phi^n_1\simeq 0$ via the homotopy
\begin{align*}H(f_0[f_1|\cdots |f_k]):=\\\sum (-1)^i\iota_\alpha\eta_{\alpha\beta}A_1(f_0)\pi_1[\iota_1A_1(f_1)\pi_1|\cdots|\iota_2| A_2(f_i)\pi_2|\cdots ]	.
\end{align*}
Here, the sum is over all the ways in which we can start with $A_1(f_0)$ and then change parity at precisely $n$ positions and then we've inserted $\iota_2$ after the last position that has $A_1$ in it. Similarly $\phi_2^n\simeq 0$ for $n\geq 1$. This proves the lemma.

\end{proof}

Suppose we have three dg categories with endofunctors $(\catC,F), (\catC',F')$ and $(\catC'',F'')$ and $(\phi,\eta):(\catC,F)\to (\catC',F')$, $(\psi,\epsilon):(\catC',F')\to (\catC'',F'')$. In this situation we define
$$(\psi,\epsilon)\circ (\phi,\eta):=(\psi\circ \phi,\epsilon\star \eta)$$
where $\epsilon\star\eta$ is the composite
$$\begin{tikzcd}
	\psi\circ \phi\circ F\arrow{r}{\psi(\eta)}&\psi\circ F'\circ \phi\arrow{r}{\epsilon_\phi}&F''\circ \psi\circ\phi.
\end{tikzcd}$$
\begin{lemma}\label{functoriality}
	We have $\big((\psi,\epsilon)\circ(\phi,\eta)\big)_*=(\psi,\epsilon)_*\circ (\phi,\eta)_*$.
\end{lemma}
\begin{proof}
	They are just equal already at the chain level.
\end{proof}
Now suppose we have $(\phi,\eta):(\catC,F)\to (\catC',F')$ and $\alpha:\phi\overset{\sim}{\implies}\psi$ is an isomorphism of functors. Then by abuse of notation we will denote by $\alpha\eta\alpha^{-1}$ the natural transformation $\psi\circ F\implies F'\circ \psi$ defined as the composite
$$\begin{tikzcd}
	\psi F\arrow{r}{\alpha_F^{-1}}&\phi\circ F\arrow{r}{\eta}&F'\circ \phi\arrow{r}{F'(\alpha)}&F'\circ \psi.
\end{tikzcd}$$
\begin{lemma}\label{invariance}
	We have $(\phi,\eta)_*=(\psi,\alpha\eta\alpha^{-1})_*$.
\end{lemma}
\begin{proof}
	The following homotopy proves the equality at the level of homology
	\begin{align*}H(f_0[f_1|\cdots |f_k]):=\\\sum_{i=0}^{k} (-1)^{i}\eta\alpha^{-1}\psi(f_0)[\psi(f_1)|\cdots |\psi(f_i)|\alpha|\phi(f_{i+1})|\cdots|\phi(f_k)]	.
	\end{align*}

\end{proof}

\subsection{Künneth isomorphism}
In this section we give a proof of the Künneth isomorphism for Hochschild homology and check that it is compatible with the $S_2$ action.
Given two dg categories $\catC$ and $\catB$ we define 
	$$Sh:C_\bullet(\catC)\otimes C_\bullet(\catB)\to C_\bullet(\catC\otimes \catB)$$
	by the formula
	\begin{align*}
f_0[f_1|\cdots |f_k]\otimes g_0[g_1|\cdots |g_l]\mapsto 	\quad\quad\quad\quad\quad\quad\quad\quad\quad\quad\quad\quad\quad\quad\quad\quad\quad\quad\quad\quad\quad\quad\\\sum_{\sigma\in (k,l)-shuffles}(-1)^{|\sigma|+\epsilon} f_0\otimes g_0[\cdots |f_{i}\otimes 1|\cdots|1\otimes g_j|\cdots|f_{i+1}\otimes 1|\cdots |1\otimes g_j|\cdots].
\end{align*}
The sign in the formula above is given by 
$$\epsilon=\sum |g_t||f_s|$$
where the sum is over all pairs $(f_s,g_t)$ so that $f_s$ appears after $g_t$.
\begin{prop}\label{kunneth}
The map $Sh$ defined above induces an isomorphism
$$HH_\bullet(\catC)\otimes HH_\bullet(\catB)\to HH_\bullet(\catC\otimes \catB).$$

When $\catC=\catB$ this isomorphism commutes with the $S_2$ action. \end{prop}

Before we prove the proposition we recall some facts about Hochschild homology. Just like for algebras Hochschild homology can be defined in a "basis independent" way as $\Delta_\catC^r\overset{L}{\otimes }\Delta_{\catC}^l$. As we show in the lemma below, the standard complex is then obtained by computing this derived tensor product using the bar resolution of $\Delta_\catC^l$ described in section \ref{backgroundondg}.
\begin{lemma}
For each n we have an isomorphism of chain complexes
$$C_n(\catC)\cong \Delta_{\catC}^r\otimes_{\catC\otimes \catC^{op}} \text{Bar}^n(\catC).$$
Moreover, this isomorphism is compatible with the Hochschild and Bar differentials and so it gives an isomorphism
$$C_\bullet(\catC)\cong Tot^{\oplus}(\Delta_\catC^r\otimes_{\catC\otimes\catC^{op}}Bar^\bullet(\catC)).$$	
\end{lemma}
\begin{proof}
	The complex on the right is the cokernel of 
	$$\bigoplus_{(x,y),(x',y')}\catC(x,y)\otimes(\catC(y',y)\otimes \catC(x,x'))\otimes \text{Bar}^n(\catC)(y',x') \to\bigoplus_{(x,y)}\catC(x,y)\otimes \text{Bar}^n(y,x)$$
	$$f\otimes (\alpha\otimes \beta)\otimes B\mapsto (-1)^{|\alpha||f|}\alpha f\beta\otimes B-(-1)^{(|B|+|\beta|)|\alpha|}f\otimes \beta B\alpha.$$
	The isomorphism is given induced by
	$$\bigoplus_{(x,y)}\catC(x,y)\otimes \text{Bar}^n(y,x)\to C_n(\catC)$$
	$$f\otimes b_0|b_1|\cdots|b_{n+1}\mapsto (-1)^{|b_{n+1}|(|b_0|+...+|b_n|+|f|)}b_{n+1}fb_0[b_1|\cdots b_n]$$
	and its inverse sends $a_0[a_1|\cdots|a_n]$ to the image of $a_0\otimes 1|a_1|...|a_n|1$ in the cokernel.
\end{proof}

If $M$ is a right (or left) $\catC$ module and $N$ is a right (or left) $\catB$ module then we can form their tensor product over $k$ which will be a right (or left) module over $\catC\otimes \catB$
$$(M\otimes_k N)(c,b):=M(c)\otimes N(b).$$
Note that tensor product of representable modules is again representable. More generally if $M$ and $N$ are projective modules of the special type described in lemma \ref{projectivemodules} then their tensor product is again of that form.
\begin{lemma}
Let $\catC$ and $\catB$ be two small dg categories. We have a homotopy equivalence of left $\catC\otimes \catB\otimes \catC^{op}\otimes \catB^{op}$-modules
$$Sh':\text{Bar}(\catC)\otimes \text{Bar}(\catB)\to \text{Bar}(\catC\otimes \catB).$$	
\end{lemma}
\begin{proof}
The map is defined by 
\begin{align*}\Big(\bigoplus_{c_1,...,c_{p+1}}\catC(c_1,-)\catC(c_2,c_1)\cdots \catC(-,c_{p+1})\Big)\otimes \Big(\bigoplus_{b_1,...,b_{q+1}}\catB(b_1,-)\catC(b_2,b_1)\cdots \catB(-,b_{q+1})\Big)\\
\longrightarrow \bigoplus_{(c_1,b_1),...,(c_{p+q+1},b_{p+q+1})}\catC\catB((c_1,b_1),-)\catC\catB((c_2,b_2),(c_1,c_1))\cdots \catC\catB(-,(c_{p+q+1},b_{p+q+1}))\\
f_0|f_1|\cdots| f_p\otimes g_0| g_1|\cdots | g_q 
\longmapsto \\(-1)^{\epsilon_1} \sum(f_0\otimes g_0 )| \cdots|(f_i\otimes1)|\cdots |(1\otimes g_j)|\cdots  |(f_{i+1}\otimes 1)|\cdots |(1\otimes_{g_{j+1}})|\cdots| (f_p\otimes g_q)
\end{align*}
where the sum is over all $(p-1,q-1)-suffles$ $\sigma.$
 The signs are given by
\begin{align*}&\epsilon_1=|g_0|(|f_1|+...+|f_p|)+|f_p|(|g_1|+...+|g_q|)\\&\epsilon_2= |\sigma|+\sum |f_i||g_j|
\end{align*}
where the sum is over all pairs $(f_i,g_j)$ such that $g_j$ appears before $f_i$.
This is a chain map which fits into a diagram
$$\begin{tikzcd}
\text{Bar}(\catC)\otimes \text{Bar}(\catB)\arrow{d}{\sim}\arrow{r}{Sh'}&\text{Bar}(\catC\otimes \catB)\arrow{d}{\sim}\\
\catC\otimes \catB \arrow{r}{=}&	\catC\otimes \catB
\end{tikzcd}
$$
Because both $\text{Bar}(\catC)\otimes \text{Bar}(\catB)$ and $\text{Bar}(\catC\otimes \catB)$ form resolutions of $\catC\otimes \catB$ by projective  $\catC\otimes \catB\otimes \catC^{op}\otimes \catB^{op}$-modules $Sh'$ is a homotopy equivalence. 
\end{proof}
\begin{proof}[Proof of proposition \ref{kunneth}]
	The isomorphism is obtained by tensoring the homotopy equivalence $Sh'$ from the previous lemma on the left with $\Delta_{\catC\otimes \catB}^r$.
\end{proof}

\section{Hochschild homology of equivariant category}
Now we again consider an action of a finite group $G$ on a small dg category $\catC$. Let $g\in G$ and denote by $C(g)$ the subgroup of elements that commute with $g$. Then, for any $h\in C(g)$ we have a natural isomorphism
$$C_{h,g}:=\theta_{g,h}^{-1}\circ\theta_{h,g}:\rho_h\circ\rho_g\implies \rho_g\circ \rho_h$$
and they induce isomorphisms
$$(\rho_h,C_{h,g})_*:HH_\bullet(\catC,\rho_g)\overset{\sim}{\to}HH_\bullet(\catC,\rho_g).$$
\begin{lemma}
The assignment $h\mapsto (\rho_h,C_{h,g})_*$ defines a right action of $C(g)$ on $HH_\bullet(\catC,\rho_g).$	
\end{lemma}
\begin{proof}
We have
\begin{align*}(\rho_h,C_{h,g})\circ(\rho_{h'},C_{h',g})=(\rho_h\circ\rho_{h'},(C_{h,g})_{\rho_{h'}}\circ\rho_{h}(C_{h',g})).
\end{align*}
and $\theta_{h,h'}:\rho_{h}\circ\rho_{h'}\overset{\sim}{\implies}\rho_{h'h}$ and $\theta_{h,h'}(C_{h,g}\star C_{h',g})\theta_{h,h'}^{-1}=C_{h'h,g}$. It then follows from lemmas $\ref{functoriality}$ and $\ref{invariance}$ that $(\rho_h,C_{h,g})_*\circ (\rho_{h'},C_{h',g})_*=(\rho_{h'h},C_{h'h,g})_*$.

\end{proof}

The tensor product functor $\text{Rep}(G)\times \catC^G\to \catC^G$ makes $HH_\bullet(\catC^G)$ into a module over the ring $R_k(G):=R(G)\otimes_\Z k$ where $R(G)$ is the representation ring of $G$. When $k=\C$ it follows from character theory that this ring is nothing but the ring of complex valued functions on the set of conjugacy classes of $G$. Denote by $\varphi_{[g]}\in R_\C(G)$ the indicator function on the conjugacy class of $g$. Note that the $\varphi_{[g]}$ form a collection of pairwise orthogonal idempotents which sum to the identity. Therefore, any $R_\C(G)$-module $M$ decomposes as $\bigoplus_{[g]\in G/\sim}\varphi_{[g]}\cdot M$.

Now we are ready to state and prove our main theorem.
\begin{theorem}\label{mainthm}
	Let $G$ be a finite group which acts on a small dg category $\catC$ and assume that $\catC$ has finite direct sums. Then we have
	$$HH_\bullet(\catC^G)\cong\bigoplus_{[g]\in G/\sim}HH_\bullet(\catC;\rho_g)^{^{C(g)}}$$
	where $\sim$ refers to conjugacy. Moreover, when the ground field is $\C$, this isomorphism identifies $HH_\bullet(\catC;\rho_g)^{C(g)}$ with the submodule $\varphi_{[g]}\cdot HH_\bullet(\catC^G)\subset HH_\bullet(\catC^G)$.
\end{theorem}

\begin{proof}
We will first define inclusions and projections $\iota_g, \pi_g$ and show that they give rise to a decomposition $HH_\bullet(\catC^G)\cong \bigoplus_{[g]\in G/\sim} HH_\bullet(\catC,\rho_g)^{C(g)}$	. Then we check that the image of $\iota_g\circ \pi_g$ is precisely the submodule $\varphi_{[g]}HH_\bullet(\catC^G)\subset HH_\bullet(\catC^G)$.

Let $For:\catC^G\to \catC$ and $S:\catC\to \catC^G$ be the functors introduced in section \ref{Finite group actions on dg categories}. We define $\pi_g=(For,\alpha_g)_*:HH_\bullet(\catC^G)\to HH_\bullet(\catC;\rho_g)$ where (by abuse of notation) $\alpha_g$ denotes the natural isomorphism $For\overset{\sim}{\implies }\rho_g\circ For$ whose value at an equivariant object $(c,\alpha)$ is $\alpha_g:c\to \rho_g(c)$.

\textit{Claim 1:} The map $\pi_g$ lands in $HH_\bullet(\catC;\rho_g)^{C(g)}$.\\
\textit{Proof of claim 1:} 
We consider $(\rho_h,C_{h,g})\circ (For,\alpha_g)=(\rho_h\circ For,C_{h,g}\star\alpha_g)$ and note that $\alpha_h:For\overset{\sim}{\implies}\rho_h\circ For$ and $C_{h,g}\star\alpha_g=\alpha_h \alpha_g\alpha_h^{-1}$. The claim then follows from lemmas \ref{functoriality} and \ref{invariance}.

Let us denote by $\hat{S}:=For\circ S$ and recall from proposition \ref{SFor} that it is isomorphic (or in fact equal) to $\bigoplus_{g\in G}\rho_g$. Now define $\iota_g:HH_\bullet(\catC;\rho_g)^{C(g)}\to HH_\bullet(\catC^G)$ by 
$$HH_\bullet(\catC;\rho_g)^{C(g)}\to HH_\bullet(\catC;\rho_g)\overset{S_g}{\to}HH_\bullet(\catC^G)$$
where $S_g=(S,\phi_{g})_*$ and $\phi_g:S(\rho_g(x_0))\to S(x_0)$ is the natural isomorphism 
$$(\phi_g)_{hh'}=\delta_{gh'=h}\theta_{h',g}.$$\\

\textit{Claim 2:} $\pi_g\circ \iota_g\simeq |C(g)|\cdot \id$.\\
\textit{Proof of claim 2:} Note that $\pi_g\circ S_g=(For\circ S,\alpha_g\star\phi_g)$
and $\alpha_g\star\phi_g:\hat{S}(\rho_g(x_0))\to \rho_g(\hat{S}(x_0))$ is give by $(\alpha_g\star\phi_g)_{hh'}=\delta_{hg=gh'}\theta_{g,h}^{-1}\circ \theta_{h',g}$. By lemma \ref{tracedecomposition} we see that $\pi_g\circ \tilde{S}=\sum_{h\in C(g)}(\rho_h,C_{h,g})_*$ and from this the claim follows.\\

\textit{Claim 3}: $\pi_g\circ \iota_{g'}\simeq 0$ if $g\nsim g'$.\\
This also follows from lemma \ref{tracedecomposition}.\\

\textit{Claim 4:} If $g$ and $g'$ are conjugate then $\iota_g\pi_g\simeq\iota_{g'}\pi_{g'}$.\\
\textit{Proof of claim 4:}
Suppose $g'=h^{-1}gh. $Consider the diagram 
$$\begin{tikzcd}
	&&(\catC,\rho_g)\arrow{dd}{(\rho_h,\theta_{g',h}^{-1}\circ \theta_{h,g})}\arrow{rrd}{(S,\phi_g)}&&\\
	(\catC^G,\id)\arrow{rru}{(F,\alpha_g)}\arrow{rrd}{(F,\alpha_{g'})}&&&&(\catC^G,\id)\\
	&&(\catC,\rho_{g'})\arrow{rru}{(S,\phi_{g'})}.&&
\end{tikzcd}$$
 We have a natural isomorphism $\alpha_h:F\implies \rho_h\circ F$ and
$$(\rho_h,\theta_{g',h}^{-1}\circ \theta_{h,g})\circ (F,\alpha_g)=(\rho_h\circ F,\rho_{g'}(\alpha_h)\circ \alpha_{g'}\circ \alpha_h^{-1}).$$

Also, we have a natural isomorphism $\phi_h:S\circ \rho_h\implies S$ and 
$$(S,\phi_{g'})\circ (\rho_h,\theta_{g',h}^{-1}\circ \theta_{h,g})=(S\circ \rho_h, \phi_h^{-1}\circ \phi_g\circ (\phi_h)_{\rho_g}).$$
Now using lemmas $\ref{functoriality}$ and $\ref{invariance}$ we get
\begin{align*}\iota_g\pi_g &=(S,\phi_g)_*\circ (F,\alpha_g)_*\\&=(S,\phi_{g'})_*\circ  (\rho_h,\theta_{g',h}^{-1}\circ \theta_{h,g})_*\circ (F,\alpha_g)_*\\	
&=(S,\phi_{g'})_*\circ (F,\alpha_{g'})_*=\iota_{g'}\circ \pi_{g'}.
\end{align*}

\textit{Claim 5:}
$$\sum_{[g]\in G/\sim}\frac{\iota_g\circ\pi_g}{|C(g)|}=\id_{HH_\bullet(\catC^G)}$$\\
\textit{Proof of claim 5:} By the previous claim we have
\begin{align*}
\sum_{[g]\in G/\sim}\frac{\iota_g\circ\pi_g}{|C(g)|}&=\sum_{[g]\in G/\sim}|conj(g)|\frac{\iota_g\circ\pi_g}{|G|}\\
&=\frac{1}{|G|}\sum_{g\in G}\iota_g\circ \pi_g.
\end{align*}
Next we note that
\begin{align*}
	\iota_g\circ \pi_g=(S,\phi_g)_*\circ(For,\alpha_g)_*=(S\circ For,\phi_g\star \alpha_g)_*
\end{align*}
and therefore 
\begin{align*}
	\sum_{g\in G}\iota_g\circ \pi_g=(S\circ For,\sum_{g\in G}\phi_g\star\alpha_g).
\end{align*}
Moreover 
$$\sum_{g\in G}\phi_g\star\alpha_g:S\circ For\implies S\circ For$$
is given in matrix form by 
$$\Big(\sum_{g\in G}\phi_g\star\alpha_g\Big)_{h'h}=\alpha_{h'}\circ \alpha_h^{-1}=(I\circ P)_{h'h}$$
where  
$$I:\id_{\catC^G}\implies S\circ For, \ P:S\circ For\implies \id_{\catC^G}$$
are the inclusion- and projection-natural transformation associated with the identity component of $\triv \otimes (-)\subset k[G]\otimes (-)$ (see proposition \ref{SFor}). Let $S'=S\circ For$. Now the following homotopy completes the proof of the claim.
\begin{align*}H(f_0[f_1|\cdots |f_k)=\\
\sum_{i=0}^k (-1)^{i}P\circ S'(f_0)[S'(f_0)|\cdots |S'(f_i)|I|f_{i+1}|\cdots f_k] 	.
\end{align*}

It remains to prove the second part of the theorem, that when the ground field is $\C$ we have $\iota_g\circ\pi_g(HH_\bullet(\catC^G))=\varphi_{[g]}\cdot HH_\bullet(\catC^G)$.
Since we know now that $HH_\bullet(\catC^G)=\bigoplus_{[g]\in G/\sim}\iota_g\circ\pi_g(HH_\bullet(\catC^G))$ and $HH_\bullet(\catC^G)=\bigoplus_{[g]\in G/\sim}\varphi_{[g]}\cdot HH_\bullet(\catC^G))$ it suffices to prove that $\iota_g\circ\pi_g(HH_\bullet(\catC^G))\subset\varphi_{[g]}\cdot HH_\bullet(\catC^G)$ for each conjugacy class $[g]$. Note that for any $R_\C(G)$ module $M$ the we have that
$$\varphi_{[g]}\cdot M=\{m\in M|[V]\cdot m=\chi_V(g)\cdot m\text{ for any representation }V\}.$$
Let $(V,\rho_V)$ be a representation with basis $\{e_1,...,e_n\}$. Let $T_V:=V\otimes(-):\catC^G\to \catC^G$. Then, multiplication by $[V]$ on Hochschild homology is the same as $(T_V,1)_*$. But note that using proposition \ref{SFor} we have the following isomorphism of functors
\begin{align}\label{decomp}T_V\circ S\circ For\cong T_V\circ T_{k[G]}\cong T_{k[G]}\circ T_V=S\circ For\circ T_V\cong \oplus_{i=1}^nS\circ For.
\end{align}
Call the composite of these natural isomorphisms $\tau$. It is given in matrix form by
$$\tau_{(i,k),(j,h)}=\delta_{h=k}\rho_V(k)_{ij}$$ 
where $1\leq i,j\leq n$ and $h,k\in G$. Let $I_j:S\circ For\implies T_V\circ S\circ For$ and $P_j:T_V\circ S\circ For\implies S\circ For$ be the inclusion and projection-natural transformations associated with the decomposition (\ref{decomp}) above. We have
$$P_j\circ T_V(\phi_g\star \alpha_g)\circ I_j=\rho_V(g)_{jj}\cdot(\phi_g\star\alpha_g).$$
We then have
\begin{align*}
	(T_V,1)_*\circ \iota_g\circ \pi_g=&(T_V,1)_*\circ (S,\phi_g)_*\circ (For,\alpha_g)_*\\
	=& (T_V\circ S\circ For,T_V(\phi_g\star \alpha_g))_*\\
	=&\sum_{j=1}^n\rho_V(g)_{jj}\cdot (S\circ For,\phi_g\star\alpha_g)_*\\
	=&\chi_V(g)( \iota_g\circ \pi_g)
\end{align*}
where the second equality is lemma \ref{functoriality} and the third equality combines the decomposition (\ref{decomp}) above and lemma \ref{tracedecomposition}.

\end{proof}

\subsection{The case of symmetric powers}

Now for any dg category $\catC$ we can define its tensor powers. Recall that by definition $\catC^{\otimes n}$ is the dg category with objects $x_\bullet=(x_1,x_2,...,x_n)$ with each $x_i\in \Ob(\catC)$. The homomorphism spaces are given by $Hom_{\catC^{\otimes n}}(x_\bullet,y_\bullet):=\catC(x_1,y_1)\otimes \catC(x_2,y_2)\otimes...\otimes \catC(x_n,y_n)$.
There is a right action of $S^n$ on $\catC$ defined on objects by $\sigma(x_1,...,x_n)=(x_{\sigma(1)},...,x_{\sigma(n)}).$ This action is strict in the strongest possible sense. All the higher categorical data is trivial, $\theta_{\sigma,\tau}=\id$, $\eta=\id$. Let $\widehat{\catC^{\otimes n}}$ be the pre triangulated hull of $\catC^{\otimes n}$ and note that $S_n$ still acts on $\widehat{\catC^{\otimes n}}$. If we set $\Sym^n(\catC)=(\widehat{\catC^{\otimes n}})^{S_n}$ then the theorem tells us
$$HH_\bullet(\Sym^n(\catC))=\bigoplus_{[\sigma]\in S_n/\sim } HH_\bullet\big(\widehat{\catC^{\otimes n}},\rho_\sigma\big)^{C(\sigma)}.$$
In particular it contains one summand (the one corresponding to the identity permutation)
$$HH_\bullet\big(\widehat{\catC^{\otimes n}}\big)^{S_n}\cong HH_\bullet(\catC^{\otimes n})^{S_n}\cong (HH_\bullet(\catC)^{\otimes n})^{S_n}=\Sym^n(HH_\bullet(\catC))$$
where the second isomorphism is the Künneth isomorphism (which is an isomorphism of $S_n$-modules). We get the following corollary.
\begin{cor}
	If $\catC$ is a small dg category then $HH_\bullet(\Sym^n(\catC))$ contains naturally $\Sym^n(HH_\bullet(\catC))$ as a summand.
\end{cor}

\end{document}